\documentclass[]{article}
\usepackage{amsmath}
\usepackage[nolist]{acronym}
\usepackage{amsthm}
\usepackage{amssymb}
\usepackage{pstool}
\usepackage{dsfont}
\usepackage{subfig}
\usepackage{todonotes}
\usepackage[T1]{fontenc}
\usepackage{graphviz}
\usepackage{easyReview}
\usepackage{url}
\usepackage{hyperref}

\newtheorem{theorem}{Theorem}

\newtheorem{lemma}{Lemma}
\theoremstyle{definition}
\newtheorem{definition}{Definition}

\renewcommand{\vec}{\mathbf}
\newcommand{\set}[1]{\left\{#1\right\}}
\newcommand{\nb}{\text{nb}}
\newcommand{\abs}[1]{\left|#1\right|}
\renewcommand{\deg}{\text{deg}}

\newcommand{\eps}{\varepsilon}	

\newcommand{\N}{\mathds{N}}

\newcommand{\UNN}{\mathrm{UNN}}

\title{On Unique Neighborhoods in Bipartite and Expander Graphs}
\author{Stefan Rass\thanks{LIT Secure and Correct Systems Lab, Johannes Kepler University Linz, email: \texttt{stefan.rass@jku.at}}}

\begin{document}
	\maketitle
	
	\begin{abstract}
		An undirected graph is said to have \emph{unique neighborhoods} if any two distinct nodes have also distinct sets of neighbors. In this way, the connections of a node to other nodes can characterize a node like an ``identity'', irrespectively of how nodes are named, as long as two nodes are distinguishable. We study the uniqueness of neighborhoods in (random) bipartite graphs, and expander graphs. 
	\end{abstract}

\section{Introduction}

This note is an extension to \cite{rassTellMeWho2024} on \acp{UNN}, where we study some two additional classes of graphs, namely bipartite and expander graphs for having (or not having) the \ac{UNN} property. The main results are the facts that random bipartite graphs with edges only between two subsets $V,W$ almost surely exhibit unique neighborhoods, if $V$ and $W$ are asymptotically of the same cardinalities. For expander graphs, we show that unique neighborhoods are logically independent of the graph to be or not to be an expander; more precisely, the \ac{UNN} property neither follows from or is refutable from any value of the graph's Cheeger constant.




\begin{definition}[\acl{UNN} \cite{rassTellMeWho2024}]\label{def:unn}
	A graph $G=(V,E)$ is a \ac{UNN}, if no neighborhood is a subset of another node's neighborhood. That is, for every $v$, there is some $u\in\nb(v)$ with $u\notin\nb(w)$ for all $w\neq v$, or equivalently, there is a nonempty symmetric difference $\nb(u)\triangle\nb(v)\neq\emptyset$ for all distinct $u,v$.
\end{definition}

To ease our notation, we let the set of all \acp{UNN} be denoted by $\UNN$, meaning that $G\in\UNN$ is a shorthand for unique neighborhoods in $G$, and $G\notin\UNN$ implying that at least two nodes in $G$ have the same neighborhood. More formally, the definition of a \ac{UNN} is the following:

\section{\acp{UNN} in Random Bipartite Graphs}

We start from \cite[equation (2)]{rassTellMeWho2024}, 
\[
B_{ij} := (\vec A\cdot \vec 1_{N\times N}-\vec A^2)_{ij}=\sum_{k=1}^N a_{ik}\cdot 1 - \sum_{k=1}^N a_{ik}a_{kj} = \sum_{k=1}^N a_{ik}(1-a_{kj}).
\]
where we put $N=n+m$ as the total number of nodes in the graph, and let $\vec A$ be its (usual) adjacency matrix\footnote{for bipartite graphs, the bi-adjacency matrix of $n\times m$ shape is an occasional alternative, but explicitly not used here}.

Therein, $a_{ik}=1$ if $i\in V_1$ and $k\in V_2$ become connected, which happens with a constant probability $0<p<1$. Thus, $\Pr(a_{ik}=1)=\Pr(i\in V_1\land k\in V_2\land S)$, if $S$ is the event of this edge to appear, with $\Pr(S)=p$. The selection of $i$ and $k$ is made independently and uniformly from $V_1, V_2$, and likewise is $S$ independent of both, and has the same probability for all possible edges. From this, we get 
\begin{align*}
	\Pr(a_{ik}=1)&=\Pr(i\in V_1)\cdot\Pr(k\in V_2)\cdot p\\
	&=\frac n{n+m}\cdot \frac{m}{n+m}\cdot p
\end{align*}
so that 
\begin{align*}
	\Pr(a_{ik}=1\land a_{kj}=0)&=\left( \frac n{n+m}\cdot \frac{m}{n+m}\cdot p\right) \cdot \left( 1-\frac n{n+m}\cdot \frac{m}{n+m}\cdot p\right)\\ &=:q
\end{align*}
Now, we have all terms inside the definition of $B_{ij}$ be values in $\set{0,1}$
considering that $\Pr(B_{ij}=0)=(1-q)^N$ by a binomially distributed choice of zero from $B_{ij}\in\set{0,1,\ldots,N}$. It remains to show that $q$ is bounded between 0 and 1, and this will depend on the growth of $\abs{V_1}=n$ relative to $\abs{V_2}=m$. For example, if we let $N\to\infty$, while keeping either $n$ or $m$ bounded, unique neighborhoods are impossible by the pigeon hole principle. However, if we assume $\abs{V_1}\in\Theta(\abs{V_2})$, then there are two constants $a,b>0$ such that for sufficiently large $m,n$, we have $\abs{V_1}\leq a\cdot\abs{V_2}$ and $\abs{V_2}\leq b\cdot \abs{V_1}$, which puts $q$ strictly inside the unit interval $(0,1)$. From this point onwards, the proof carries over verbatim from  \cite[equation (3)]{rassTellMeWho2024}, giving the conclusion that the bipartite graph will grow into a \ac{UNN}. This proves the following result:

\begin{theorem}
Let a random bipartite graph $G_{n,m}$ be given with $n=\abs{V_1}, m=\abs{V_2}$ and $n\in\Theta(m)$. Let the edges of $G$ appear with a probability $p$ that does not depend on $n$ or $m$. Then, as $n+m\to\infty$, $G_{m,n}$ will almost surely be a \ac{UNN}.
\end{theorem}

\section{\acp{UNN} in Expander Graphs}
We will show the logical independence of unique neighborhoods from the value of the graph's Cheeger constant. To this end, let us briefly recall the respective definitions.

Let $G=(V,E)$ be a graph with $n=\abs{V}$ vertices. For a subset $S\subseteq V$, write $\partial S=\set{\{u,v\}: u\in S, v\notin S}$ for the set of edges that connect a node inside $S$ with a node outside $S$. The Cheeger constant $h(G)$ is defined as $h(G)=\min_{S\subset V: \abs{S}\leq\abs{V}/2}\frac{\abs{\partial S}}{\abs{S}}$. We call $G$ an \emph{$(n,\eps,d)$-expander}\footnote{We remark that the definition of expander families is not entirely consistent throughout the literature, since other authors denote the relevant parameters in different order or form}, if all nodes in $G$ have a degree $\leq d$, and the graph has a Cheeger constant $h(G)\geq\eps$. If $\lambda_2$ is the second-largest eigenvalue of the normalized Laplacian matrix of the graph, which is $\vec L=\vec I-\vec D^{-1/2}\cdot \vec A\cdot \vec D^{1/2}$, where $\vec D=\text{diag}(\deg(v_1),\ldots,\deg(v_n))$ is the diagonal matrix with node degrees, and $\vec A$ is the graph's adjacency matrix. The celebrated Cheeger-inequality states that $\lambda_2\leq 2 h(G)\leq\sqrt{2\lambda_2}$. The usual interest is in infinite sets of expander graphs with constant $d$ and $\eps$, and we denote such a set by $\mathcal{G}(N,\eps,d)$, including graphs of all sizes $n\in N\subseteq\N$, but all with the same $d$ and $\eps$, and with $\abs{N}=\infty$.

Specifically, let $\eps$ be arbitrary but fixed, then we will construct example graphs for the following four cases. For each case, we let $n>1$ be arbitrary, implying that there is an infinitude of graphs with the given properties.

\begin{table}
	\caption{Logical independence of \acp{UNN} from the size of the graph's Cheeger constant}\label{tbl:cheeger-unn-independence}
	\begin{tabular}{|c||p{4cm}|p{4cm}|}
		\hline 
		& Cheeger constant $\geq\varepsilon$ & Cheeger constant $<\varepsilon$\tabularnewline
		\hline 
		\hline 
		\ac{UNN} & complete graph; Lemma \ref{lem:large-cheeger-unn} & circle graph; Lemma \ref{lem:small-cheeger-unn}\tabularnewline
		\hline 
		not a \ac{UNN} & complete bipartite graph; Lemma \ref{lem:no-unn-large-cheeger} & any graph $G(V,\emptyset)$ with $V\neq\emptyset$; see also Lemma \ref{lem:no-unn-small-cheeger}\tabularnewline
		\hline 
	\end{tabular}
	
\end{table}

\begin{lemma}\label{lem:large-cheeger-unn}
For every $\eps>0$ there is some $n>1$ and a graph $G_n\in\UNN$ of size $n$ that has a Cheeger constant $h(G_n)\geq \eps$.
\end{lemma}
\begin{proof}
Simply consider the complete graph $K_n=(V,V\times V)$ with $\abs{V}=n$, which is a \ac{UNN} \cite[Lemma 5]{rassTellMeWho2024}.

For $A\subseteq V$, with size $\abs{A}$, we have $\abs{\partial{A}}=\abs{A}\cdot (n-\abs{A})$ many edges outgoing from $A$, giving the $\frac{\abs{A}\cdot(n-\abs{A})}{\abs{A}}=n-\abs{A}$, which is at its minimum for $\abs{A}=n/2$, giving the Cheeger constant $h(G)=n/2$. The result follows by making $n>2\eps$.
\end{proof}

\begin{lemma}\label{lem:small-cheeger-unn}
For every $\eps>0$ there is some $n\in\N$ and a graph $G_n\in\UNN$ of size $n$ that has a Cheeger constant $h(G_n)<\eps$.
\end{lemma}
\begin{proof}
The sought graph is the circle graph $C_n$ with $n$ nodes. This graph is a \ac{UNN} \cite[Lemma 5]{rassTellMeWho2024}. Its Cheeger constant is easily to find by considering the set $A=\set{1,2,\ldots,\lfloor n/2\rfloor}$, having outgoing edges $\set{\{\lfloor n/2\rfloor, \lfloor n/2\rfloor+1 \},\{n,1\}}$, which gives $\frac{\partial A}{\abs{A}}=\frac 2{\lfloor n/2\rfloor}\to 0$ as $n\to\infty$. Thus, for sufficiently large $n$, the Cheeger constant will become $<\eps$.
\end{proof}

\begin{lemma}\label{lem:no-unn-large-cheeger}
For every $\eps>0$ there is some $n>1$ and a graph $G_n\notin\UNN$ of size $n$ with a Cheeger constant $h(G_n)\geq\eps$.
\end{lemma}

\begin{proof}
The sought graph is the complete bipartite graph $K_{n/2,n/2}$ for even $n$. This is a well-known example of a Ramanujan-graph, whose spectral gap is largest possible; specifically, with the second-smallest eigenvalue $\lambda_2$ of the graph's Laplacian, we have a diverging spectral gap, if we let $d=n/2\to\infty$. The divergence of the Cheeger constant $h(K_{n/2,n/2})\to\infty$ then follows from Cheeger's inequality.
\end{proof}

\begin{lemma}\label{lem:no-unn-small-cheeger}
For every $\eps>0$ there is some $n>1$ and a graph $G_n\notin\UNN$ of size $n$ with a Cheeger constant $h(G_n)<\eps$.
\end{lemma}
\begin{proof}
	Given $V\neq\emptyset$, the graph $G=(V,\emptyset)$ without edges is trivially not a \ac{UNN}, since all neighborhoods are identically $\emptyset$ and has $h(G)=0<\eps$ as desired.
	
	An alternative argument working with a connected graph is the following: fix any even integer $n\geq 6$ and take the circle graph $C_{n-2}=(V,E)$ with $V=\set{1,\ldots,n}$. Then, add two extra nodes $n+1, n+2$ with edges $\set{1,n+1},\set{1,n+2},\set{2,n+1}$, and $\set{2,n+2}$, and call the resulting graph $G_n$. The neighborhoods of nodes 1 and 2 in $G_n$ are equally $\set{n+1,n+2}$ and vice versa. Hence, $G_n\notin\UNN$. The Cheeger constant is easily upper bounded by choosing $A=\set{3,4,\ldots,n/2+2}\subset V\setminus\set{1,2,n+1,n+2}$. This set has $\abs{A}= n/2$ and $\abs{\partial A}=2$ (the only outgoing edges are to nodes $2$ and $n/2+3$), making $0<h(G)\to 0$ as we let $n$ grow beyond all limits.
\end{proof}

We can summarize the results in Table \ref{tbl:cheeger-unn-independence} as a compact statement:
\begin{theorem}
	Let $G$ be a graph. Whether $G$ is a \ac{UNN} is neither provable nor refutable from any bound on $G$'s Cheeger constant. 
\end{theorem}  

While the counterexamples in Table \ref{tbl:cheeger-unn-independence} are all graphs of fixed size and structure, the \ac{UNN}-property is also independent of the graph's membership to an expander-family.


Recall the \emph{2-lift} construction of expander-graphs as devised in \cite{bilu_lifts_2006}: Fix some $d$-regular graph $G=(V,E)$, called the \emph{base graph}. Consider a function $s:E\to\set{-1,+1}$, which is called a \emph{signing}. A \emph{$2$-lift} $\hat G$ of $G$ w.r.t. the signing $s$ is created by associating two vertices $x_1,x_2$ with every vertex $x\in V$, and adding edges in $\hat G$ as follows: for each $\set{x,y}\in E$, add the edges $\{\set{x_1,y_1}, \set{x_2,y_2}\}$ to $\hat{G}$ if $s(x,y)=+1$. Otherwise, if $s(x,y)=-1$, then add the edges $\{\set{x_1,y_2},\set{x_2,y_1}\}$ to $\hat G$.

This operation has been modified later in \cite{marcus_interlacing_2015} to construct $d$-regular expanders of all sizes. 

\begin{lemma}\label{lem:2-lift-unn}
The 2-lift of a \ac{UNN} is again a \ac{UNN}.
\end{lemma}
\begin{proof}
	Let the base graph be a \ac{UNN}. Borrowing an argument from the proof of \cite[Lemma 3.1]{bilu_lifts_2006}, the adjacency matrix of the 2-lifted graph $\hat G$ can be written as a block matrix
	\[
	\hat {\vec A}=\left(\begin{array}{c|c}
		{\vec A}_1 & {\vec A}_2\\\hline
		{\vec A}_2 & {\vec A}_1
	\end{array}\right),
	\]
	where ${\vec A}_1$ is the adjacency matrix of $(V,s^{-1}(1))$ and ${\vec A}_2$ is the adjacency matrix of $(V,s^{-1}(-1))$. The base graphs' adjacency matrix is ${\vec A}={\vec A}_1+{\vec A}_2$. If $G\in\UNN$, then the rows in ${\vec A}$ are unique, so towards a contradiction, assume that $\hat G\notin\UNN$. This means that at least one row in $\hat {\vec A}$, say $i$, is duplicate with a row $j$, except on the respective diagonal positions (we do not count a node as its own neighbor). To exclude the diagonal, let $J=\set{1,2,\ldots,n}\setminus\set{i,j}$, and write $({\vec A}_1)_{i,J}$, resp. $({\vec A}_2)_{j,J}$ to mean all elements in the rows $i$ and $j$ with column indices from $J$. 
	
	First, observe that we can take $1\leq i\leq n$ and $n+1\leq j\leq 2n$, since if $i$ and $j$ were both in ${\vec A}_1$ or both in ${\vec A}_2$, then we would have an identical neighborhood with neighbors in ${\vec A}_1$, but also an identical neighborhood with entries from ${\vec A}_2$. The union of the two would therefore also be the same in the base graph already, contradicting the assumption that $G\in\UNN$. The following figure illustrates the argument, where we write $(\vec A)_{i,\bullet}$ to denote the $i$-th entire row of a matrix ${\vec A}$:
	
	\begin{figure}[h!]
		\centering
		\includegraphics[scale=0.66]{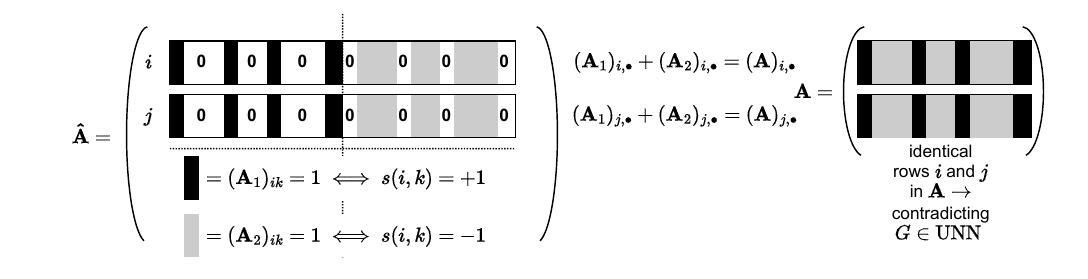}
	\end{figure}
	
	Therefore, we can assume $i$ as a row in ${\vec A}_1$ and $j$ as a row in ${\vec A}_2$. Since the rows have identical entries in the columns in $J$, the first half of row $j$ must also appear as the second half in row $i$ (by definition of $\hat {\vec A}$). Again, since $\hat G\notin\UNN$, the second half of row $i$ in $\hat {\vec A}$ must therefore also be the second half in row $j$ in $\hat {\vec A}$. Now, consider any index $k\in J$, for which $({\vec A}_1)_{ik}=1\iff s(i,k)=+1$. Since ${\vec A}_2$ has also a 1-entry in row $j$ column $k$, this means $s(j,k)=1$. But this further means that $({\vec A}_2)_{i,k}=1$, since the left half of row $j$ appears as right half of row $i$, implying that $s(i,k)=-1$ by definition of ${\vec A}_2$; a contradiction. The figure below illustrates the argument similar as before, highlighting the position of the $k$-th column in light gray:
	
	\begin{figure}[h!]
		\centering
		\includegraphics[scale=0.66]{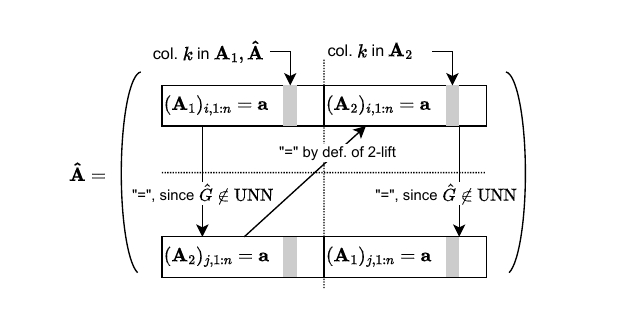}
	\end{figure}

	This means that the 2-lift of a base graph $G\in\UNN$ will produce a graph $\hat G\in\UNN$, and completes the proof.
\end{proof}

\begin{theorem}
Let $\mathcal{G}(N,\eps,d)$ be an expander-family. Then, the statements $G\in\mathcal{G}$ and $G\in\ac{UNN}$ are logically independent.
\end{theorem}
\begin{proof}
We construct expander-families that have, or do not have, the \ac{UNN} property, starting with expander families whose families are all \acp{UNN}, and later modifying them to remain expanders, but losing the \ac{UNN} property. 

Since Lemma \ref{lem:2-lift-unn}'s assertion about the preservation of the \ac{UNN} property upon 2-lifts made no particular assumptions on the underlying signing, it will work for all choices of $s$. Hence, the known constructions of expander families as in \cite{bilu_lifts_2006} remain doable, and we directly get an expander family composed of all \acp{UNN}.

To construct an expander that is no longer a \ac{UNN}, consider an $(n,\eps,d/2)$-expander $G$. If $G\notin\UNN$ already, then we are finished. Otherwise, choose any two distinct nodes $x,y$ in $G$, and add edges from $x$ to $\nb(y)$ and from $y$ to $\nb(x)$. Call the resulting graph $G'$. By construction, it is not a \ac{UNN}, but still an expander: The Cheeger constant cannot decrease upon this change, since any subset of nodes will have at least as many, if not more, outgoing edges after this change, making $h(G')\geq h(G)\geq\eps$. Moreover, the maximum degree in $G'$ is no more than twice the degree as in $G$, since every two nodes in $G$ have at most $d/2$ neighbors, and even if the neighborhoods were disjoint (not just distinct), adding another $d/2$ many edges to both of them is enough to make the neighborhoods non-unique. This gives a $(n,\eps,d)$-expander $G'\notin\UNN$, and can be repeated for all sizes $n$, while keeping $\eps$ and $d$ constant.

\end{proof}

\begin{acronym}
	\acro{UNN}{Unique-Neighborhood Network}%
	\acro{ISP}{Internet Service Provider}%
	\acro{MPT}{Multipath Transmission}%
	\acro{MPA}{Multipath Authentication}%
	\acro{MAC}{Message Authentication Code}%
	\acro{IoT}{Internet-of-Things}%
\end{acronym}

\bibliographystyle{plain}

\end{document}